\documentclass{amsart}    % Specifies the document style.
\usepackage{amsthm, amsmath, amscd, amssymb}

\usepackage[textwidth=27mm, textsize=small]{todonotes}
\usepackage{ulem}

\usepackage{xcolor}
\usepackage[all]{xy}
\input xypic

\theoremstyle{plain}
\newtheorem{theorem}{Theorem}[section]
\newtheorem{lemma}[theorem]{Lemma}
\newtheorem{proposition}[theorem]{Proposition}
\newtheorem{corollary}[theorem]{Corollary}

\newtheorem*{thm}{Theorem}
\theoremstyle{definition}
\newtheorem{definition}[theorem]{Definition}
\newtheorem{remark}[theorem]{Remark}
\newtheorem{rk-def}[theorem]{Remark-Definition}
\newtheorem{notation}[theorem]{Notation}
\newtheorem{example}[theorem]{Example}
\theoremstyle{remark}

\newtheorem*{ack}{Acknowledgement}
\usepackage[mathscr]{eucal}
\usepackage{graphics, graphpap}
\usepackage{array, tabularx, longtable}

\usepackage[top=3.0cm, bottom=3.0cm, left=3.0cm, right=3.0cm]{geometry}

\numberwithin{equation}{section}

\def\Nash{\mathbf{Nash}}
\def\Grass{\mathbf{Grass}}
\def\Hilb{\mathbf{Hilb}}

\def\Jac{\mathrm{Jac}}
\def\JI{\mathcal{J}}

\def\Crit{\mathrm{Crit}}
\def\Fitt{\mathbf{Fitt}}
\def\div{\mathrm{div}}

\def\pr{\mathrm{pr}}
\def\Var{\mathrm{Var}}
\def\Mor{\mathrm{Mor}}

\def\Im{\mathrm{Im}}

\def\ord{\mathrm{ord}}

\def\Spec{\mathrm{Spec}}

\def\Id{\mathrm{Id}}

\def\sr{\mathrm{sr}}

\def\L{\mathbb{L}}

\title[Higher Jacobian ideals, contact equivalence and motivic zeta functions]{\bf Higher Jacobian ideals, contact equivalence and motivic zeta functions}
\author{Quy Thuong L\^e}
\address{University of Science, Vietnam National University, Hanoi
\newline 
\indent 334 Nguyen Trai street, Hanoi, Vietnam}
\address{Department of Mathematics, Graduate School of Science, Osaka University, Toyonaka
\newline\indent Osaka 560-0043, Japan}
\email{leqthuong@gmail.com}

\author{Takehiko Yasuda}
\address{Department of Mathematics, Graduate School of Science, Osaka University, Toyonaka
\newline\indent Osaka 560-0043, Japan}
\email{yasuda.takehiko.sci@osaka-u.ac.jp}

\thanks{The first author was partially funded by VinGroup and supported by Vingroup Innovation Foundation (VinIF) under the project code VINIF.2021.DA00030.
The second author
was supported by JSPS KAKENHI Grant Numbers JP21H04994
and JP23H01070.
}

\begin{document}
\begin{abstract}
%When generalizing \sout{Yasuda's} higher Nash blowup to that of any morphism of varieties it arises a standard sheaf of ideals which plays the role of a center of an isomorphic blowup. In this article, working with regular functions on smooth varieties, we develop some essential relations of this sheaf of ideals with contact equivalence of complex analytic germs and motivic zeta functions. 
%
%[[A more specific version?]]
We show basic properties of higher Jacobian matrices and higher Jacobian ideals for  functions and apply it to obtain two results concerning singularities of functions. Firstly, we prove that a higher Nash blowup algebra is invariant under contact equivalences, which was recently conjectured by Hussain, Ma, Yau and Zuo. Secondly, we obtain an analogue of a result on motivic nearby cycles by Bussi, Joyce and Meinhardt.
\end{abstract}
\maketitle

\section{Introduction}

% \todo[inline]{Takehiko (Aug.\ 31): I have realized that the \(n\)-th Jacobian matrix \(\Jac_n(f)\) in \cite{BD20} is the same as our \(\bJac_n(f)\). Should we stop using two different versions?\\
% Thuong (Aug. \ 31): You are right, they are the same. \\
% Thuong (Sep. 5th): Now I agree with you, we will remove $\circ$ in the notation of higher Jacobian matrices and ideals, but make a revision of Subsection 2.3, put a remark (before Corollary \ref{conjj}) about notation to compare with the conjecture of Hussain et. al.} 
% \todo{Takehiko (Aug.\ 30): I try to write a new subsection which may replace the current subsection 1.1. Higher Nash blowup does not play an important role any longer in this paper. I put more emphasis on other subjects.}

\subsection{}
Jacobian matrices and Jacobian ideals are fundamental objects in the study of singularities of varieties as well as singularities of morphisms. 
There are higher versions of these notions, higher Jacobian matrices and higher Jacobian ideals. 
Duarte introduced them for hypersurface varieties in \cite{Dua17}, motivated by the study of higher Nash blowups. The case of more general varieties was later treated in \cite{BD20,BJNB19}. The aim of the present article is to show basic properties of higher Jacobian matrices and ideals for functions/morphisms and apply them to show two results concerning singularities of functions, that is, a recent conjecture on higher Nash blowup algebras by Hussain, Ma, Yau and Zuo \cite{HMYZ23} as well as an analogue of a result on motivic neaby cycles by Bussi, Joyce and Meinhardt \cite{BJM19}. 

Throughout this article, we work over a field $k$  of characteristic zero. 
By a $k$-variety we mean a separated and reduced scheme of finite type over $k$. For a $k$-variety $S$, by $S$-variety we mean a variety $X$ together with a morphism $X\to S$. 
For a non-negative integer $n$ and an $S$-variety $X$, we have the coherent $\mathcal O_X$-module $\mathcal P_{X/S}^n$ and $\Omega_{X/S}^{(n)}$ that are called the sheaf of principal parts of order $n$ and the sheaf of K\"ahler differentials of order $n$, respectively.
These sheaves are closely related to the higher Nash blowup. This blowup was studied, for example, in [[here we give serveral references]] mainly from the viewpoint of the desingularization problem.  
For a \(k\)-variety \(X\), its \(n\)-th Nash blouwp, denoted by  \(\Nash_n(X)\), can be defined as the blowup \(\mathrm{Bl}_{\mathcal P_{X/k}^n}(X)=\mathrm{Bl}_{\Omega_{X/k}^{(n)}}(X)\)
associated to \(\mathcal P_{X/k}^n\) or \(\Omega_{X/k}^{(n)}\).  Basic properties of blowups associated to general coherent sheaves were studied previously in  
\cite{OZ91,Vil06}. It is natural to consider a generalization of this blowup from \(k\)-varieties to \(S\)-varieties.
We may define the $n$-th Nash blowup $\Nash_n(X/S)$ of $X$ over $S$ to be $\mathrm{Bl}_{\mathcal P_{X/S}^n}(X)=\mathrm{Bl}_{\Omega_{X/S}^{(n)}}(X)$. We may speculate that this blowup would shed new light on the study of singularities of morphisms \(X\to S\)  or ones of the corresponding foliations on \(X\). 
(Although a further generalization to foliations with non-algebraic leaves would be also interesting, we do not treat it in this paper.)
The study of higher Nash blowups of \(S\)-varieties itself is not the main subject of this paper, this motivates as the study of sheaves \(\mathcal P_{X/k}^n\) and \(\Omega_{X/k}^{(n)}\).

Our principal interest is in the case where \(S=\mathbb A_k^1\) and $X$ is smooth. Focusing on the even more specialized situation \(X=\mathbb A_k^d\) is also important, as the local study of the given function $f: X\to \mathbb A_k^1$ is reduced to this situation. Our first small result is 
to verify an analogue of a result proved in \cite{Dua17, BD20,BJNB19}. Namely, we observe that for a morphism \(f:\mathbb A_k^d \to \mathbb A_k ^1\), there exists a free presentation 
% \todo{Thuong (Sep. 6th): I added the powers in the exact sequence. I also had some slight revisions.}
\[
\mathcal O_{\mathbb A_k^d}^{\binom{d-1+n}{d}} \to \mathcal O_{\mathbb A_k^d}^{\binom{d+n}{d}-1} \to 
\Omega_{\mathbb A_k^d /\mathbb A_k^1}^{(n)}\to 0
\]
such that the left map is given by a higher Jacobian matrix \(\Jac_n(f)\) defined in \cite{BD20} (for the singule function case), which has higher derivatives of \(f\) and zeroes as entries. Restricting the above exact sequence to the hypersurface defined by \(f\) refines an earlier result in \cite{Dua17}, and is similar to the exact sequence independently obtained in \cite[Corollary 2.27]{BJNB19}. This free presentation shows that the \(n\)-th Jacobian ideal sheaf \(\JI_n(f)\) is generated by the maximal minors of \(\Jac_n(f)\). Combining this with  a result of Villamayor \cite{Vil06} also gives that for a function \(f:X \to \mathbb A_k^1\) on a smooth variety, the \(n\)-Nash blowup \(\Nash_n(X/\mathbb A_k^1)\) coincides with the blowup with respect to the ideal \(\JI_n(f)\); this is an analogue of Duarte's result for hypersurfaces \cite[Proposition 4.11]{Dua17}.

\subsection{}
Consider $k=\mathbb C$ and the ring $\mathbb C\{x\}=\mathbb C\{x_1,\dots,x_d\}$ of complex analytic functions on a neighborhood of $\mathbf 0$ in $\mathbb C^d$. In \cite{HMYZ23}, Hussain, Ma, Yau and Zuo use the definition of higher Jacobian matrices introduced earlier in \cite{Dua17} and work with the ideal of $\mathbb C[x]$ generated by all maximal minors of such a higher Jacobian matrix of a polynomial (to compare different versions of higher Jacobian matrices, see Remark \ref{versions-Jac}). 

Let $f$ be in $\mathbb C[x]$ such that $f(\mathbf 0)=0$. The $n$-th Nash blowup local algebra of $f$ is defined as follows $\mathcal M_n(f):=\mathbb C\{x\}/\langle f, \JI_n(f)\rangle$. It follows from Remark \ref{versions-Jac} that this definition is independent of the choice of versions mentioned previously of higher Jacobian matrices. Clearly, $\mathcal M_1(f)$ is the Milnor algebra. The $n$-th Nash blowup local algebra $\mathcal M_n(f)$ and the derivation Lie algebra $\mathcal L_n(f):=\mathrm{Der}(\mathcal M_n(f))$ are main objectives studied in \cite{HMYZ23}. In terms of \cite[Conjecture 1.5]{HMYZ23} it was expected that if $f\in \mathbb C[x]$ defines an isolated singularity at $\mathbf 0$, then $\mathcal M_n(f)$ is a contact invariant (this was checked in \cite{HMYZ23} with $n=d=2$). We go further when checking the conjecture with arbitrary $n, d\in \mathbb N^*$, especially we do not require the isolatedness of singularities; namely, we have

\begin{thm}[Theorem \ref{local-thm1}]
Let $f$ and $g$ be in $\mathbb C[x]$ with $f(\mathbf 0)=g(\mathbf 0)=0$. If $f$ is contact equivalent to $g$ at $\mathbf 0$, then $\mathcal M_n(f)$ is isomorphic to $\mathcal M_n(g)$ as $\mathbb C$-algebras for any $n\in \mathbb N^*$.	
\end{thm}

\subsection{}
Let $X$ be a smooth $k$-variety. Let $f$ be a non-constant regular function on $X$, and let $f$ also denote the corresponding element in $\mathcal O_X(X)$. As mentioned above, to such an $f$ associate the $\mathcal O_X$-module $\Omega_f^{(n)}$, the Nash blowup $\Nash_n(f)$, and the sheaf of ideals $\JI_n(f)\subseteq \mathcal O_X$. Now, we associate to $f$ the motivic zeta function $Z_f(T)$ and nearby cycles $\mathscr S_f$. It would be interesting to realize an essential relation between $\JI_n(f)$ and $Z_f(T)$ as well as $\mathscr S_f$. In fact, the latter strongly motivates our present work.

Motivic zeta function is a profound incarnation of applications of motivic integration to singularity theory. Indeed, several singularity invariants such as Hodge-Euler characteristic and Hodge spectrum can be recovered from motivic zeta functions via appropriate Hodge realizations (see e.g. \cite{DL98, DL02}). Since the motivic zeta function relates directly to the monodromy conjecture, it has been widely taken care by several geometers and singularity theorists. For instance, Denef-Loeser (cf. \cite{DL98, DL02}) give it an explicit description using log-resolution, hence they can prove its rationality and list a candidate set of poles concerning the log-resolution numerical data, which is very important for any approach to the conjecture. Moreover, the motivic zeta function and the motivic nearby cycles are also important in mathematical physics, especially in the study of motivic Donalson-Thomas invariants theory for noncommutative
Calabi–Yau threefolds (cf. \cite{BJM19}).

For any integer $m\geq 1$, let $K_0^{\mu_m}(\Var_S)$ denote the $\mu_m$-equivariant Grothendieck ring of $S$-varieties endowed with a good $\mu_m$-action. Localizing $K_0^{\mu_m}(\Var_S)$ with respect to the class $\L$ of the trivial line bundle over $S$ obtains $\mathscr M_S^{\mu_m}$, and taking inductive limit of $\mathscr M_S^{\mu_m}$ with respect to $m$ gives the ring $\mathscr M_S^{\hat\mu}$ (see Section \ref{Sec3.1}). Write 
$$Z_f(T)=\sum_{m\geq 1}\big[\mathscr X_m(f)\big]\L^{-dm}T^m,$$ 
where $\mathscr X_m(f)$ is the $m$-th iterated contact locus of $f$ defined in Section \ref{Sec3.1}. The following theorem is also a main result of this article. 

\begin{thm}[Theorem \ref{theorem1}]
Let $f$ and $g$ be non-constant regular functions on $X$ with the same scheme-theoretic zero locus $X_0$. Suppose that $g-f\in \JI_2(f)$. If $d=2$ and $X_0$ has nodes, suppose additionally that $k$ is quadratically closed. Then, for any integer $m\geq 1$, the identity $\big[\mathscr X_m(f)\big]=\big[\mathscr X_m(g)\big]$ holds in $\mathscr M_{X_0}^{\mu_m}$. As a consequence, $Z_f(T)=Z_g(T)$ and $\mathscr S_f=\mathscr S_g$.
\end{thm}

The theorem points out the dependence of the motivic zeta function and the motivic nearby cycles of $f$ on the second order Jacobian ideal sheaf $\JI_2(f)$. We can compare it with the ones obtained by Bussi-Joyce-Meinhardt \cite[Theorems 3.2, 3.6]{BJM19}, in which instead of $\JI_2(f)$ they concern the ideal sheaf $\JI_1(f)^3$. The proof in \cite{BJM19} uses in a crucial way Proposition 4.3 of \cite{BBDJS15}. Our proof uses $m$-separating log-resolution whose existence is given in \cite{BFLN22}. 

%Furthermore, we can strengthen this result by removing the assumption that $f$ and $g$ have the same zero locus $X_0$ (see Theorem \ref{strengthen}).
%

%***********************

\section{Higher Jacobian matrices and higher Jacobian ideals}

\subsection{Higher Nash blowups}
We first recall the construction of Oneti-Zatini \cite{OZ91} on the Nash blowup associated to a coherent sheaf. Let $X$ be a reduced Noetherian scheme, and let $\mathcal M$ be a coherent $\mathcal O_X$-module locally of constant rank $r$ on an open dense subscheme $U$ of $X$. Consider the functor $\mathcal G$ from the category of $X$-schemes to the category of sets which sends each $X$-scheme $Y$ to the set of $\mathcal O_Y$-modules that are locally free quotient of rank $r$ of $\mathcal M\otimes_{\mathcal O_X}\mathcal O_Y$. Then $\mathcal G$ is contravariant and represented by an $X$-scheme $\Grass_r\mathcal M$. It is a fact that $\Grass_r\mathcal M \times_XU$ is isomorphic to $U$. 

\begin{definition}[Oneto-Zatini]
The closure of $\Grass_r\mathcal M \times_XU$ is called the {\it blowup of $X$ at $\mathcal M$}, and denoted by $\mathrm{Bl}_{\mathcal M}(X)$. 
\end{definition}

The natural morphism $\pi_{\mathcal M}: \Grass_r\mathcal M \to X$ is birational as $X$ is reduced, and it is projective as $\mathcal M$ is coherent. Note that $(\pi_{\mathcal M}^*\mathcal M)/\mathrm{Tor}(\pi_{\mathcal M}^*\mathcal M)$ is locally free. As shown in \cite{OZ91}, $\mathrm{Bl}_{\mathcal M}(X)$ has the universal property, namely, if $h: Y\to X$ is a modification such that $(h^*\mathcal M)/\mathrm{Tor}(h^*\mathcal M)$ is locally free, then there exists a unique morphism $\phi: Y\to \mathrm{Bl}_{\mathcal M}(X)$ such that $\pi_{\mathcal M}\circ \phi=h$.

Let $S$ be a reduced Noetherian scheme, and $f: X\to S$ be a morphism of schemes. The diagonal morphism 
$$\Delta=\Delta_f: X\to X\times_SX$$ 
of $f$ is locally a closed immersion, so $\Delta(X)$ is closed in an open subset $V$ of $X\times_SX$. Let $\mathcal I=\mathcal I_f$ be the ideal sheaf defining $\Delta(X)$ in $V$, i.e. $\mathcal I$ is the kernel of 
$$\Delta^{\#}: \mathcal O_{X\times_SX} \to \Delta_*\mathcal O_X.$$ 
For any $n\in \mathbb N^*$, we define
$$\mathcal P_f^n:=\mathcal P_{X/S}^n:=\mathcal O_{X\times_SX}/\mathcal I^{n+1}$$
and
$$\Omega_f^{(n)}:=\Omega_{X/S}^{(n)}:=\mathcal I/\mathcal I^{n+1}.$$ 
As explained in \cite[Sections 16.3, 16.4]{Gro67}, $\mathcal P_f^n$ (hence $\Omega_f^{(n)}$) depends functorially on $f$, and the construction is local, i.e. for any open subscheme $U$ of $X$ we have $\mathcal P_{f|U}^n=\mathcal P_f^n|_U$ and $\Omega_{f|U}^{(n)}=\Omega_f^{(n)}|_U$. Let $p_1$ and $p_2$ be canonical projections of $X\times_SX$. Any of these two morphisms defines a homomorphism of sheaves of rings $d_i:\mathcal O_X\to \Omega_f^{(n)}$ ($i=1, 2$) for every $n\in \mathbb N$. Thus there are two structures of $\mathcal O_X$-algebras on $\Omega_f^{(n)}$; we shall fix the $\mathcal O_X$-algebra structure induced by $p_1$. 

\begin{definition}
The $\mathcal O_X$-modules $\mathcal P_f^n$ and $\Omega_f^{(n)}$ are called the {\it sheaf of principal parts of order $n$} of $f$ and the {\it module of K\"ahler differentials of order $n$} of $f$, respectively.
\end{definition}

Let us consider morphisms of schemes $f: X\to S$ and $g: S\to B$. From the commutative diagram of schemes
$$
\begin{CD}
X @>\Id>> X\\
@V\Delta_fVV @VV\Delta_{g\circ f}V\\
X\times_SX @>\ell>> X\times_BX
\end{CD}
$$
we have the following commutative diagram of structural sheaves
$$
\begin{CD}
\mathcal O_{X\times_BX} @>\ell^*>> \mathcal O_{X\times_SX}\\
@V\Delta_{g\circ f}^{\#}VV @VV\Delta_f^{\#}V\\
(\Delta_{g\circ f})_*\mathcal O_X @>>> (\Delta_f)_*\mathcal O_X
\end{CD}
$$
Thus $\ell^*(\mathcal I_{g\circ f})\subseteq \mathcal I_f$, and in general for any $n\in \mathbb N$, $\ell^*(\mathcal I_{g\circ f}^n)\subseteq \mathcal I_f^n$. This fact yields the canonical morphism of sheaves of $\mathcal O_X$-algebras
\begin{align}\label{Theta}
\Theta_n: \Omega_{g\circ f}^{(n)}\to \Omega_f^{(n)}.
\end{align}
We now assume that $\Delta_f$, $\Delta_g$ and $\ell: X\times_SX\to X\times_BX$ are closed immersions. Let $\mathcal K$ be the sheaf of ideals of $\mathcal O_{X\times_BX}$ corresponding to $\ell$. Then, as explained in the proof of \cite[Proposition 16.4.18]{Gro67}, 
$$\mathcal K=(f\times f)^*(\mathcal I_g)\cdot \mathcal O_{X\times_BX} \subseteq \mathcal I_{g\circ f}$$ 
and $\Theta_n$ is nothing but the canonical projection
$$\Omega_{g\circ f}^{(n)}\to \Omega_{g\circ f}^{(n)}\big/\big((\mathcal I_{g\circ f}^{n+1}+\mathcal K)/\mathcal I_{g\circ f}^{n+1}\big).$$
Similarly, from the commutative diagram
$$
\begin{CD}
X @>f>> S\\
@V\Delta_{g\circ f}VV @VV\Delta_gV\\
X\times_BX @>f\times f>> S\times_BS
\end{CD}
$$
we have the canonical morphisms of sheaves of $\mathcal O_X$-algebras, for any $n\in \mathbb N$,
$$\Psi'_n: f^*\Omega_g^{(n)}=\Omega_g^{(n)}\otimes_{\mathcal O_S}\mathcal O_X\to \Omega_{g\circ f}^{(n)}$$
and
$$\Psi''_n: \mathcal K\hookrightarrow f^*\Omega_g^{(n)}\to \Omega_{g\circ f}^{(n)}.$$
Let $\Psi_n$ be the homomorphism $\mathcal K\otimes_{\mathcal O_X}\mathcal P_{g\circ f}^n\to \Omega_{g\circ f}^{(n)}$ induced by $\Psi''_n$, namely, $\Psi_n$ is the composition of 
$$\Psi''_n\otimes \Id: \mathcal K\otimes_{\mathcal O_X}\mathcal P_{g\circ f}^n\to \Omega_{g\circ f}^{(n)}\otimes_{\mathcal O_X}\mathcal P_{g\circ f}^n$$ 
with the product homomorphism 
$$\Omega_{g\circ f}^{(n)}\otimes_{\mathcal O_X}\mathcal P_{g\circ f}^n\to \Omega_{g\circ f}^{(n)}.$$ 
Then $\Im(\Psi_n)$ is nothing but the sheaf of ideals generated by $\Psi'_n(\mathcal K)=\Im(\Psi''_n)$.

\begin{lemma}\label{exact}
With the previous notation, $\Theta_n$ is surjective and $\ker(\Theta_n)=\Im(\Psi_n)$. As a consequence, there is an exact sequence of sheaves of $\mathcal O_X$-modules
$$
\begin{CD}
\mathcal K\otimes_{\mathcal O_X}\mathcal P_{g\circ f}^{n-1} @>\Psi_n>> \Omega_{g\circ f}^{(n)} @>\Theta_n>> \Omega_f^{(n)} @>>> 0.
\end{CD}
$$	
\end{lemma}

\begin{proof}
From  \cite[Proposition 16.4.18]{Gro67}, the image of $\mathcal K$ in $\Omega_{g\circ f}^{(n)}$ generates the kernel of $\Theta_n$ as a $\mathcal P_{g\circ f}^{n}$-submodule. Namely, the natural map 
$$\Psi_n: \mathcal K\otimes_{\mathcal O_X}\mathcal P_{g\circ f}^{n} \to \Omega_{g\circ f}^{(n)}$$ 
is surjective. We need to show that this factors through 
$$\mathcal K\otimes_{\mathcal O_X}\mathcal P_{g\circ f}^n
\twoheadrightarrow
\mathcal K\otimes_{\mathcal O_X}\mathcal P_{g\circ f}^{n-1}.$$
To do so, we first observe that  $\mathcal K \subseteq \mathcal I_{g\circ f}$. 
If $a$ is a section of $\mathcal K$ and if $b$ is a section of $\mathcal I_{g\circ f} ^n/\mathcal I_{g\circ f} ^{n+1}$, then the section $a\otimes b$ of 
$\mathcal K\otimes_{\mathcal O_X}\mathcal P_{g\circ f}^{n}$ 
maps to the section $ab = 0 $ of
$\Omega_{g\circ f}^{(n)}=\mathcal I_{g\circ f} /\mathcal I_{g\circ f} ^{n+1} $. Thus, we get the desired factorization.
\end{proof}

\begin{notation}\label{notation}
Let $x=(x_1,\dots,x_d)$ be an ordered family of $n$ variables, and $f$ in $k[x]$. Let $\alpha=(\alpha_1,\dots,\alpha_d)$ and $\beta=(\beta_1,\dots,\beta_d)$ be in $\mathbb N^d$. From now on, we use the following notation
\begin{itemize}
\item[(i)] $\alpha!=\alpha_1!\cdots \alpha_d!$, $|\alpha|=\alpha_1+\cdots+\alpha_d$,

\item[(ii)] $\alpha\geq \beta$ (equivalently, $\beta\leq \alpha$) if $\alpha_i\geq \beta_i$ for all $1\leq i\leq d$,\\ 
$\alpha>\beta$ (equivalently, $\beta<\alpha$) if $\alpha\geq \beta$ but $\alpha\not=\beta$,

\item[(iii)] $x^{\alpha}=x_1^{\alpha_1}\cdots x_d^{\alpha_d}$, and 
$$\partial^{\alpha}f=\frac{\partial^{|\alpha|}f}{\partial x_1^{\alpha_1}\cdots \partial x_d^{\alpha_d}}.$$
\end{itemize}	
\end{notation}

\begin{example}\label{keyexample}
Let $x=(x_1,\dots,x_d)$, $x'=(x'_1,\dots,x'_d)$ and $R=k[x]$. Let $f\in R$, which defines a morphism $f: X\to \mathbb A_k^1$. Consider the special case $X=\mathbb A_k^d=\Spec R$, $S=\mathbb A_k^1=\Spec k[t]=\Spec k[t']$ and $B=\Spec k$. Then we have 
$$X\times_kX=\Spec k[x,x'],\ \ S\times_kS=\Spec k[t,t'],$$ 
and
$$\mathcal I_{g\circ f}=\langle x'_1-x_1,\dots,x'_d-x_d\rangle,\ \ \mathcal I_g=\langle t-t'\rangle,\ \ \mathcal K=\langle f(x')-f(x)\rangle.$$ 
We obtain furthermore that
$\mathcal P_{g\circ f}^{n-1}\cong R^{{{d+n-1}\choose{d}}}$ and its canonical $R$-basis is 
$$\{[(x'-x)^{\beta}]\mid 0\leq |\beta|\leq n-1\};$$ 
similarly, $\Omega_{g\circ f}^{(n)}\cong R^{{{d+n}\choose{d}}-1}$ and its canonical $R$-basis is $\{[(x'-x)^{\alpha}]\mid 1\leq |\alpha|\leq n\}$. A system of generators of $R$-module $\mathcal K\otimes_R\mathcal P_{g\circ f}^{n-1}$ is formed by the following vectors
$$(f(x')-f(x))\otimes [(x'-x)^{\beta}], \quad 0\leq |\beta|\leq n-1.$$
By Taylor's expansion we have
$$f(x')-f(x)=\sum_{\gamma> 0}\frac{\partial^{\gamma}f(x)}{\gamma!} (x'-x)^{\gamma},$$
hence for $0\leq |\beta|\leq n-1$,
\begin{align*}
\Psi_n\big((f(x')-f(x))\otimes [(x'-x)^{\beta}]\big)&=\Big[(x'-x)^{\beta}\sum_{\gamma> 0}\frac{\partial^{\gamma}f(x)}{\gamma!} (x'-x)^{\gamma}\Big]\\
&=\sum_{1\leq |\alpha|\leq n,\ \! \alpha> \beta}\frac{\partial^{\alpha-\beta}f(x)}{(\alpha-\beta)!}[(x'-x)^{\alpha}].
\end{align*}
Using the exact sequence in Lemma \ref{exact} we have that 
$$\bigg\{\sum_{1\leq |\alpha|\leq n,\ \! \alpha> \beta}\frac{\partial^{\alpha-\beta}f(x)}{(\alpha-\beta)!}[(x'-x)^{\alpha}] \mid  0\leq |\beta|\leq n-1 \bigg\}$$
is a system of generators of $\ker(\Theta_n)$ as an $R$-module.	
\end{example}

%The following corollary follows from Lemma \ref{exact}.

Similarly as in \cite[Corollary 16.4.22]{Gro67}, we can show that if $f$ is of finite type, then $\Omega_f^{(n)}$ is a quasi-coherent sheaf on $X$ and an $\mathcal O_X$-module of finite type. As $S$ is Noetherian and $f$ is of finite type, it implies that $f$ is of finite presentation, hence $\Omega_f^{(n)}$ is an $\mathcal O_X$-module of finite presentation (this can be also seen explicitly due to Lemma \ref{exact}). 

%\begin{remark}\label{remark1}
%In this article, we consider an algebraically closed field $k$, and by an (algebraic) $k$-variety we mean a separated and reduced scheme of finite type over $k$.  
If $f$ is a morphism of varieties, we denote by $\Crit(f)$ its critical locus. 

\begin{lemma}\label{genrank}
Let $f: X\to S$ be a dominant morphism of $k$-varieties of relative dimension $e$ with $\Crit(f)$ of codimension $\geq 1$. Then $\Omega_f^{(n)}$ is a coherent $\mathcal O_X$-module locally free of constant rank $r={{e+n}\choose{e}}-1$ on an open dense subset $U$ of $X$.	
\end{lemma}

\begin{proof}
Since the statement is local, we can take the smooth part of $X$, hence it suffices to consider $X=\mathbb A_k^m$ and $S=\mathbb A_k^l$ (where $e=m-l$). We can assume further that $l=0$ (hence $m=e$). Finally, it is fact that $\mathrm{rank}\ \Omega_{\mathbb A_k^e/k}^{(n)}=r$.
\end{proof}
 
%\end{remark}
Note that if $X$ is smooth and $S=\mathbb A_k^1$, then the condition that $\Crit(f)$ has codimension $\geq 1$ is nothing but $f$ non-constant.

\begin{definition}
Let $f: X\to S$ be a morphism of $k$-varieties with $\Crit(f)$ of codimension $\geq 1$. The {\it $n$-th Nash blowup of $f$} is defined to be the blowup of $X$ at $\Omega_f^{(n)}$, $\mathrm{Bl}_{\Omega_f^{(n)}}(X)$, and denoted by $\Nash_n(f)$, or by $\Nash_n(X/S)$ when the morphism $f: X\to S$ is fixed.	
\end{definition}

Consider the particular case where $f$ is $X\to \Spec k$. Assume the $k$-dimension of $X$ is $d$. Then $\Nash_n(X/\Spec k)$ is nothing but the $n$-th Nash blowup of $X$, denoted by $\Nash_n(X)$, in the sense of Yasuda \cite{Yas07}. Denote by $\Delta(X)^{(n)}$ the $n$-th infinitesimal neighborhood of $\Delta(X)$, and by $\pr_1$ the restricted first projection $\Delta(X)^{(n)}\to X$. According to \cite{Yas07}, $\Nash_n(X)$ is the irreducible component dominating $X$ of the relative Hilbert scheme $\Hilb_{{d+n}\choose{d}}(\pr_1)$ for a constant Hilbert polynomial ${d+n}\choose{d}$. Since the moduli schemes $\Hilb_{{d+n}\choose{d}}(\pr_1)$ and $\Grass_{{d+n}\choose{d}}(\mathcal O_X\oplus \Omega_f^{(n)})$ present equivalent functors we have 
$$\Nash_n(X)\cong \mathrm{Bl}_{\mathcal O_X\oplus \Omega_f^{(n)}}(X) \cong \mathrm{Bl}_{\Omega_f^{(n)}}(X)=\Nash_n(X/\Spec k),$$ 
see \cite[Proposition 1.8, Corollary 1.9]{Yas07}.

Let $f: X\to S$ be a morphism of relative dimension $e$ with $\Crit(f)$ of codimension $\geq 1$. Let $\mathcal Q(X)$ be the sheaf of total quotient rings of $X$. Let $n\in \mathbb N^*$ and $r={{e+n}\choose{e}}-1$. Let $\psi$ be the composition of the canonical morphism 
$$\bigwedge^r\Omega_f^{(n)} \to \bigwedge^r\Omega_f^{(n)}\otimes_{\mathcal O_X}\mathcal Q(X)$$ 
and a fixed isomorphism 
\begin{align}\label{non-can}
\bigwedge^r\Omega_f^{(n)}\otimes_{\mathcal O_X}\mathcal Q(X)\to \mathcal Q(X).
\end{align}
Then $\Im\psi$ is a coherent fractional ideal of $\mathcal Q(X)$, locally free of rank $1$ on an open dense subscheme $U$ of $X$. In general, there are several isomorphisms as (\ref{non-can}), so the identification of $\bigwedge^r\Omega_f^{(n)}\otimes_{\mathcal O_X}\mathcal Q(X)$ and $\mathcal Q(X)$ is not canonical, thus $\Im\psi$ is well defined up to isomorphism.

\begin{proposition}\label{Nash-f}
Let $f: X\to S$ be a morphism of $k$-varieties of relative dimension $e$ with $\Crit(f)$ of codimension $\geq 1$, let $n\in \mathbb N^*$ and $r={{e+n}\choose{e}}-1$. There is an isomorphism of $k$-varieties 
$$\Nash_n(f)\cong \mathrm{Bl}_{\bigwedge^r\Omega_f^{(n)}}(X).$$ 
Moreover, $\Nash_n(f)$ is isomorphic to the blowup of $X$ with respect to the fractional ideal $\Im\psi$.	
\end{proposition}

\begin{proof}
This is a direct application of Oneto-Zatini's result \cite[Theorem 3.1]{OZ91} to the $k$-variety $X$ and the sheaf $\mathcal M=\Omega_f^{(n)}$.	
\end{proof}

%----------------

\subsection{Higher Jacobian ideals of regular functions}\label{Sec2.2}
Let $X$ be a scheme, and let $\mathcal M$ be a coherent $\mathcal O_X$-module. For each non-negative integer $i$,  the {\it $i$-th Fitting ideal sheaf of $\mathcal M$}, denoted by $\Fitt_i(\mathcal M)$, is defined as follows. We take an open covering $X= \bigcup _\lambda U_\lambda$ such that for each $\lambda$, $\mathcal M|_{U_\lambda}$ admits a free presentation:
$$\begin{CD}
    \mathcal O_{U_\lambda }^{N'} @>\psi_\lambda>> \mathcal O_{U_\lambda }^N @>\theta>> \mathcal M|_{U_\lambda} @>>> 0.
    \end{CD}
$$
For each $\lambda$, we define $\Fitt_i(\mathcal M)|_{U_\lambda}\subset \mathcal O _{U_\lambda}$ to be the ideal sheaf generated by $(N-i)\times(N-i)$-minors of the matrix representing $\psi_\lambda$. We see that these ideal sheaves glue together to give an ideal sheaf on the entire scheme $X$; we define $\Fitt_i(\mathcal M)$ to be this ideal sheaf. We denote by $\mathcal K$ the ideal sheaf glued from the kernel of $\theta$ varying $\lambda$.

\begin{remark}
Similarly, we can define Fitting ideals $\Fitt_i (\mathcal M)$ of a coherent sheaf $\mathcal M$ on a complex analytic space $X$.
\end{remark}

We discuss more on the local setting. Let $U$ be an affine open subscheme of $X$, $M=\mathcal M(U)$, $R=\mathcal O_X(U)$ and $K=\mathcal K(U)\subseteq R^{N'}$. Let $\{e_1,\dots,e_N\}$ be the canonical basis of $R^N$, and $\{v_1,\dots,v_N\}$ a system of generators of $M$. Let $\theta: R^N\to M$ be the $R$-homomorphism defined by $\theta(e_i)=v_i$, $1\leq i\leq N$. Let $\big\{(a_{1j},\dots,a_{Nj}) \mid 1\leq j\leq N' \big\}$ be a system of generators of $K$, and let 
$$A:=\big(a_{ij}\big)_{1\leq i\leq N, 1\leq j\leq N'}.$$ 
Then $A$ is called a {\it relation matrix of $M$ with respect to $\{v_1,\dots,v_N\}$}. Then $\mathrm{Fitt}_i(M):=\Fitt_i(\mathcal M)(U)$ is generated by all $((N-i)\times (N-i))$-minors of $A$, which is independent of the choice of $A$ and $\{v_1,\dots,v_N\}$ (cf. \cite[Lemma D.1, Lemma D.2]{Kun86}).

Let $X$ be a smooth algebraic $k$-variety of dimension $d$, and let $f: X\to \mathbb A_k^1$ be a non-constant regular function. Then $\Omega_f^{(n)}$ is a coherent $\mathcal O_X$-module locally free of constant rank $r={{d-1+n}\choose{d-1}}-1$ on an open subset of $X$ (cf. Lemma \ref{genrank}). 

\begin{definition}
The $r$-th Fitting ideal $\Fitt_r(\Omega_f^{(n)})$ of the $\mathcal O_X$-module $\Omega_f^{(n)}$ is called the {\it $n$-th Jacobian ideal of $f$} and denoted by $\JI_n(f)$.	
\end{definition}

Let us give an explicit description of $\JI_n(f)$ in the affine case $X=\Spec R$ with $R=k[x]=k[x_1,\dots,x_d]$. This case is also enough for results and applications we will mention. 

To an $f\in R$ and an $n\in \mathbb N^*$ we associate a matrix described as follows
$$\Jac_n(f):=\left(r_{\beta,\alpha}\right)_{0\leq |\beta|\leq n-1, 1\leq |\alpha|\leq n},$$
where the ordering for row and column indices is graded lexicographical, 
\begin{equation}\label{matrixJac}
r_{\beta,\alpha}=r_{\beta,\alpha}(f):=
\begin{cases}
0 & \text{if}\ \ \alpha_i<\beta_i \ \ \text{for some}\ 1\leq i\leq d\\
0 & \text{if} \ \ \alpha=\beta\\
\frac{\partial^{\alpha-\beta}f}{(\alpha-\beta)!} & \text{if}\ \ \alpha> \beta,
\end{cases}
\end{equation}
using Notation \ref{notation}. Clearly, $\Jac_n(f)$ is a matrix of type ${{d-1+n}\choose{d}}\times \big({{d+n}\choose{d}}-1\big)$ with entries in $R$.

\begin{definition}\label{bdef}
For $f\in R$, the matrix $\Jac_n(f)$ is called the {\it Jacobian matrix of order $n$ of $f$}, or the {\it $n$-th Jacobian matrix of $f$}. 
\end{definition} 

\begin{remark}\label{versions-Jac}
There are a few versions of higher Jacobian matrix which are slightly different one another. Our definition above follows the one adopted in \cite{BD20}. Another version considered in \cite{Dua17,HMYZ23} differs in that  the diagonal entries \(r_{\alpha,\alpha}\) are \(f\) instead of \(0\). These two versions coincide modulo \(f\).  The one in \cite{BJNB19}, which the authors call the {\it Jacobi-Taylor matrix}, has one extra column by allowing \(|\alpha|=0\). 
\end{remark}

Let us consider the ring homomorphism $k[t]\to R$ which maps $t$ to $f$. Let $I$ be the kernel of the diagonal homomorphism $R\otimes_{k[t]}R\to R$, which is an $R$-module via the homomorphism $R\to R\otimes_{k[t]}R$ given by $\xi \mapsto \xi\otimes 1$. By Lemma \ref{genrank}, $\Omega_f^{(n)}=I/I^{n+1}$ is an $R$-module of generic rank $r={{d-1+n}\choose{d-1}}-1$. For $1\leq i\leq d$, put $\delta x_i:=1\otimes x_i-x_i\otimes 1$, and for $\alpha=(\alpha_1,\dots,\alpha_d)$, put $(\delta x)^{\alpha}:=\prod_{i=1}^d(\delta x_i)^{\alpha_i}$. Then $I=\big\langle \delta x_1,\dots,\delta x_d \big\rangle$ and $\{[(\delta x)^{\alpha}] \mid 1\leq |\alpha|\leq n\}$ is a system of generators of the $R$-module $\Omega_f^{(n)}$. Let $\{e_{\alpha} \mid 1\leq |\alpha|\leq n\}$ be the canonical $R$-basis of $R^{{{d+n}\choose{d}}-1}$, using graded lexicographical ordering for indices $\alpha$ in $\mathbb N^d$. Consider the homomorphism of $R$-modules
$$\theta: R^{{{d+n}\choose{d}}-1}\to \Omega_f^{(n)}$$
defined by 
$$\theta(e_{\alpha})=[(\delta x)^{\alpha}].$$ 

\begin{proposition}\label{relmatrix}
With the previous notation, for $f\in R$, the transpose of $\Jac_n(f)$ is a relation matrix of $\Omega_f^{(n)}$ with respect to $\{[(\delta x)^{\alpha}] \mid 1\leq |\alpha|\leq n\}$. As a consequence, $\JI_n(f)$ is the ideal of $R$ generated by all the maximal minors of the matrix $\Jac_n(f)$.
\end{proposition}

\begin{proof}
Consider the homomorphism $\Theta_n$ defined in (\ref{Theta}) for the case $f: \mathbb A_k^d\to\mathbb A_k^1$ and $g: \mathbb A_k^1\to \Spec k$. Via the natural isomorphism $\Omega_{g\circ f}^{(n)}\stackrel{\cong}{\to} R^{{{d+n}\choose{d}}-1}$ mapping $[(x'-x)^{\alpha}]$ to $e_{\alpha}$, $\Theta_n$ is nothing but $\theta$. Again, via the previous isomorphism, it is computed in Example \ref{keyexample} that 
$$\bigg\{\sum_{1\leq |\alpha|\leq n,\ \! \alpha> \beta}\frac{\partial^{\alpha-\beta}f(x)}{(\alpha-\beta)!}e_{\alpha} \mid  0\leq |\beta|\leq n-1 \bigg\}=\bigg\{\sum_{1\leq |\alpha|\leq n}r_{\beta,\alpha}e_{\alpha} \mid  0\leq |\beta|\leq n-1 \bigg\}$$
is a system of generators of $\ker(\theta)$ as an $R$-module. This proves that the transpose of $\Jac_n(f)$ is a relation matrix of $\Omega_f^{(n)}$ with respect to $\{[(\delta x)^{\alpha}] \mid 1\leq |\alpha|\leq n\}$.
\end{proof}

\begin{example}\label{ex-Jac_2}
For $f\in R$, we have
\begin{equation*}
\Jac_2(f)=
\begin{pmatrix}
\frac{\partial f}{\partial x_1} & \frac{\partial f}{\partial x_2} & \cdots & \frac{\partial f}{\partial x_d} & \frac 1 2 \frac{\partial^2 f}{\partial x_1^2} & \frac{\partial^2 f}{\partial x_1\partial x_2} & \cdots & \frac{\partial^2 f}{\partial x_1\partial x_d} & \frac 1 2 \frac{\partial^2 f}{\partial x_2^2} & \cdots\\
0 & 0 & \cdots & 0 & \frac{\partial f}{\partial x_1} & \frac{\partial f}{\partial x_2} & \cdots & \frac{\partial f}{\partial x_d} & 0 & \cdots\\
0 & 0 & \cdots & 0 & 0 & \frac{\partial f}{\partial x_1} & \cdots & 0 & \frac{\partial f}{\partial x_2}& \cdots\\
\vdots & \vdots & \vdots & \vdots & \vdots & \vdots & \vdots & \vdots & \vdots & \vdots\\
0 & 0 & \cdots & 0 & 0 & 0 & \cdots & \frac{\partial f}{\partial x_1}& 0 & \cdots
\end{pmatrix}.
\end{equation*}
More particularly, for $f(x_1,x_2)=x_1^3-x_2^2$ we have
\begin{equation*}
\Jac_2(f)=\begin{pmatrix}
3x_1^2 & -2x_2 & 3x_1 & 0 & -1\\
0 & 0 & 3x_1^2 & -2x_2 & 0\\
0 & 0 & 0 & 3x_1^2 & -2x_2
\end{pmatrix}
\end{equation*}
and
$$\JI_2(f)=\big\langle x_1^6,\ x_1^4x_2,\ x_1^2x_2^2,\ x_2^3,\ 4x_1x_2^2-3x_1^4\big\rangle.$$
\end{example}

\begin{remark}
We have a more direct way to show that the row vectors of $\Jac_n(f)$ are elements the kernel of $\theta$. Indeed, let $r_{\beta}$ denote the $\beta$-row of $\Jac_n(f)$, namely
$$r_{\beta}=(r_{\beta,\alpha})_{1\leq |\alpha|\leq n}=\sum_{1\leq |\alpha|\leq n}r_{\beta,\alpha}e_{\alpha}\ \in R^{{{d+n}\choose{d}}-1}.$$
Then, we have 
\begin{equation*}
\begin{aligned}
\theta(r_{\beta})&=\sum_{1\leq |\alpha|\leq n,\ \! \alpha> \beta}\frac{\partial^{\alpha-\beta}f(x)}{(\alpha-\beta)!}[(\delta x)^{\alpha}]\\
&=\bigg[(\delta x)^{\beta}\sum_{1\leq |\alpha|\leq n,\ \! \alpha> \beta}\bigg(\frac{\partial^{\alpha-\beta}f(x)}{(\alpha-\beta)!}\otimes 1\bigg) \cdot (\delta x)^{\alpha-\beta}\bigg]\\
&=\bigg[(\delta x)^{\beta}\sum_{1\leq |\gamma|\leq n-|\beta|,\ \! \gamma> 0}\frac{\partial^{\gamma}f(x\otimes 1)}{\gamma!} \cdot (\delta x)^{\gamma}\bigg].
\end{aligned}
\end{equation*}
Note that $[(\delta x)^{\beta+\gamma}]=0$ in $\Omega_f^{(n)}$ for $|\beta|+|\gamma|>n$. Thus, using Taylor's expansion,
\begin{equation*}
\begin{aligned}
\theta(r_{\beta})&=\bigg[(\delta x)^{\beta}\sum_{\gamma> 0}\frac{\partial^{\gamma}f(x\otimes 1)}{\gamma!} \cdot (\delta x)^{\gamma}\bigg]\\
&=\bigg[(\delta x)^{\beta}\sum_{\gamma\geq 0}\frac{\partial^{\gamma}f(x\otimes 1)}{\gamma!} \cdot (\delta x)^{\gamma}-(\delta x)^{\beta}f(x\otimes 1)\bigg]\\
&=\big[(\delta x)^{\beta}\big(f(1\otimes x)-f(x\otimes 1)\big)\big]\\
&=\big[(\delta x)^{\beta}\big(1\otimes f-f\otimes 1\big)\big]\\
&=0.
\end{aligned}
\end{equation*}
Furthermore, the rank of $\Jac_n(f)$ is ${{d-1+n}\choose{d}}$, which is equal to ${{d+n}\choose{d}}-1-r$. This is only enough to show that $\JI_n(f)$ is isomorphic as a fractional ideal to the ideal of $R$ generated by all the maximal minors of the matrix $\Jac_n(f)$, which is weaker than the statement of Proposition \ref{relmatrix}.
\end{remark}

\begin{proposition}\label{Jn-prop}
Let $f$ be a non-constant regular function on a smooth $d$-dimensional $k$-variety $X$, and let $n\in \mathbb N^*$. Then $\Nash_n(f)$ is isomorphic to the blowup of $X$ with respect to $\JI_n(f)$.
\end{proposition}

\begin{proof}
By Proposition \ref{relmatrix} as well as \cite[Proposition 2.5, Corollary 2.6]{Vil06}, $\JI_n(f)$ and $\Im\psi$ mentioned in Proposition \ref{Nash-f} are isomorphic as fractional ideals. Thus $\Nash_n(f)$ is isomorphic to the blowup of $X$ with respect to $\JI_n(f)$.	
\end{proof}

\begin{proposition}\label{Jn-commute}
Let $f$ be a regular function on a smooth $d$-dimensional $k$-variety $X$. 
Let $\varphi: X' \to X$ be an \'etale morphism. Then, we have $\varphi^{-1}\JI_n(f)=\JI_n(f\circ \varphi)$. Similarly for the case where $f$ is a holomorphic function on a complex manifold $X$ and $\varphi: X'\to X$ is a local isomorphism of complex manifolds.
%\todo{Modified the statement and the proof. Commented out the old one.}
\end{proposition}

\begin{proof}
 Since $\varphi$ is \'etale, we have $\varphi^* \Omega_f^{(n)} = \Omega_{f\circ \varphi}^{(n)} $. From a basic property of Fitting ideals, we have  $\Fitt_r(\varphi^*\mathcal M) = \varphi^{-1}\Fitt_r(\mathcal M)$ for any coherent sheaf $\mathcal M$ on $X$. Similarly for the holomorphic setting.
\end{proof}

\begin{proposition}\label{JnJ1}
We have an inclusion $\JI_n(f)\subseteq \JI_1(f)^{\binom{d-2+n}{d-1}}$. In particular, $\JI_n(f)\subseteq \JI_1(f)^{3}$, if either 
    \begin{enumerate}
        \item[(i)] $d \ge 3$ and $n\ge 2$, or  
        \item[(ii)] $d =2$ and $n\ge 3$. 
    \end{enumerate}
\end{proposition}

\begin{proof}
If $|\beta|=n-1$, then the $\beta$-row of the matrix  $\Jac_n(f)$ has only zeroes and first derivatives $\frac{\partial f}{\partial x_i}$ as its entries. Indeed, if $\alpha$ is given by $\alpha_i=\beta_i+1$ and $\alpha_j = \beta_j$ $(i\ne j)$ for some $i$, then the $(\beta,\alpha)$-entry of the matrix is $\frac{\partial f}{\partial x_i}$. If $\alpha$ is not of this form, then the $(\beta,\alpha)$-entry is zero. 
The number of $\beta$'s with $|\beta|=n-1$ is $\binom{d-2+n}{d-1}$, and hence there are $\binom{d-2+n}{d-1}$ rows as above in the matrix.
It follows that every maximal minor of  $\Jac_n(f)$ belongs to $\JI_1(f)^{\binom{d-2+n}{d-1}}$, which shows the first assertion of the proposition.

To show the second assertion, we first note that  if $r>s$, then 
\[
\binom{r}{s} \geq \binom{r-1}{s-1}\geq \cdots \geq   \binom{r-s+1}{1} =r-s+1.
\]
If $d\geq 3$ and $n \geq 2$, then 
\[
\binom{d-2+n}{d-1} =  \binom{d-3+n}{d-2} +  \binom{d-3+n}{d-1} \geq n +(n-1) = 2 n-1 \geq 3.    
\]
If $d=2$ and $n \geq 3$, then 
\[
\binom{d-2+n}{d-1} = n \geq 3.    
\]
\end{proof}

%--------------------------

\subsection{Higher Jacobian ideals of complex analytic functions}
In this subsection, we consider the ring $\mathbb C\{x\}=\mathbb C\{x_1,\dots,x_d\}$ and $f\in \mathbb C\{x\}$ with $f(\mathbf 0)=0$, where $\mathbf 0$ is the origin of $\mathbb C^d$. Denote by $V=V(f)$ or $(V,\mathbf 0)$ the germ of the complex hypersurface singularity at $\mathbf 0$ defined by $f$. For a complex analytic function $f$, we also define $\Jac_n(f)$ and $\JI_n(f)$ similarly as in Section \ref{Sec2.2}.

\begin{lemma}\label{local-lem1}
Let $f$ be in $\mathbb C[x]$ with $f(\mathbf 0)=0$. Let $\varphi$ be an automorphism of $\mathbb C\{x\}$. Then for $n\in \mathbb N^*$,  $\varphi(\JI_n(f))=\JI_n(\varphi(f))$.
\end{lemma}

\begin{proof}
This lemma is a particular case of the local version of Proposition \ref{Jn-commute}, hence it has the same method of proof.
\end{proof}

\begin{corollary}
Let $f$ be a weighted homogeneous polynomial in $\mathbb C[x]$, and $u$ in $\mathbb C\{x\}$ with $u(\mathbf 0)\not=0$. Then for any $n\in \mathbb N^*$,  $\JI_n(f)=\JI_n(uf)$.
\end{corollary}

\begin{proof}
Let $w=(w_1,\dots,w_d)$ be the weight of $f$. Consider the automorphism $\varphi$ of $\mathbb C\{x\}$ defined by 
$$x\mapsto (u^{w_1}x_1,\dots, u^{w_d}x_d).$$
We have 
$$\varphi(f)=f(u^{w_1}x_1,\dots, u^{w_d}x_d)=uf.$$ 
Then the corollary follows from Lemma \ref{local-lem1}. 
\end{proof}

\begin{remark}
Consider $f\in \mathbb C\{x\}$ satisfying $f(\mathbf 0)=0$ and the ring homomorphism $\mathbb C\{t\}\to \mathbb C\{x\}$ sending $t$ to $f$. Let $I'$ be the kernel of the diagonal homomorphism $\mathbb C\{x\}\otimes_{\mathbb C\{t\}}\mathbb C\{x\}\to \mathbb C\{x\}$, which is an $\mathbb C\{x\}$-module due to the map $\mathbb C\{x\}\to \mathbb C\{x\}\otimes_{\mathbb C\{t\}}\mathbb C\{x\}$ sending $\xi$ to $\xi\otimes 1$. By Lemma \ref{genrank}, $I'/I^{\prime n+1}$ is an $\mathbb C\{x\}$-module of generic rank $r={{d-1+n}\choose{d-1}}-1$. Now, assume that $f$ is a polynomial in $\mathbb C[x]$ with $f(\mathbf 0)=0$, which induces a regular function $f: \mathbb A_{\mathbb C}^d\to \mathbb A_{\mathbb C}^1$. Then, similarly as \cite[Corollary 16.4.16]{Gro67}, we have $\big(\Omega_f^{(n)}\big)_{\mathbf 0}\cong I'/I^{\prime n+1}$. This also explains that some results for regular functions on varieties still hold in this setting provided $f$ is a polynomial (e.g. Lemma \ref{local-lem1}).	
\end{remark}

\begin{lemma}\label{local-lem3}
Let $f$ and $u$ be in $\mathbb C\{x\}$ such that $f(\mathbf 0)=0$ and $u(\mathbf 0)\not=0$. Then for any $n\in \mathbb N^*$,  $\langle f, \JI_n(f)\rangle=\langle f,\JI_n(uf)\rangle$.
\end{lemma}

\begin{proof}
First, we use the convention that $\partial^{\alpha-\beta}=0$ if there is an $i$ such that $\alpha_i<\beta_i$. 
For each injection
\[
    \iota: \{\beta \mid 0\le |\beta| \le n-1 \} \to \{\alpha \mid 1\le |\alpha| \le n \},
\]
we define 
$$M^f_\iota:=\big(r_{\beta,\iota(\beta')}(f)\big)_{\beta,\beta'}.$$ 
Up to permutation of columns, this is equal to the maximal square submatrix of $\Jac_n(f)$ corresponding to the image of $\iota$. Let
\begin{align*}
    T&:=\bigg(\frac{\partial^{\beta'-\beta}u}{(\beta'-\beta)!}\bigg)_{0\leq |\beta|, |\beta'|\leq n-1},\\
    N_\iota&:=T M^f_\iota.
\end{align*}
We see that $T$ is an upper triangular matrix with every diagonal entry equal to $u$. Indeed, if $\beta > \beta'$ for the graded lexicographic order, then for some $i$, $\beta_i>\beta'_i$ and hence the $(\beta,\beta')$-entry is zero.
It follows that 
\begin{align*}
\det(N_\iota )=u^{{d-1+n}\choose{d}} \det(M^f_\iota).
\end{align*}
A direct computation shows that the $(\beta,\beta')$-entry of $N_\iota$ is 
\begin{equation*}
\sum_{0\leq |\beta''|\leq n-1}\frac{\partial^{\beta''-\beta}u}{(\beta''-\beta)!}  \frac{\partial^{\iota(\beta')-\beta''}f}{(\iota(\beta')-\beta'')!},
\end{equation*}
with the convention that $\partial^{(0,\dots,0)}f=0$. From this convention and the general Leibniz rule, this is equal to $r_{\beta,\iota(\beta')}(uf)$. % if $|\iota(\beta')|\le n-1$ and equal to 
% \[
% r_{\beta,\alpha_{\beta'}}^{uf}-f\sum_{|\beta''|=n}\frac{\partial^{\beta''-\beta}u}{(\beta''-\beta)!}
% \]
%\[
%r_{\beta,\iota(\beta')}(uf)-f\frac{\partial^{\iota(\beta')-\beta}u}{(\iota(\beta')-\beta)!}
%\]
%if $|\iota(\beta')|=n$. 
Thus $M_\iota^{uf}= N_\iota$ and $\det (M_\iota^{uf})= \det(N_\iota)=\det (M_\iota^f)$. 
%\[
 %  M_\iota^{uf}\equiv N_\iota \! \mod \langle f\rangle   
%\]
%and 
%\[
 %  \det (M_\iota^{uf})\equiv \det(N_\iota)\equiv \det (M_\iota^f)\! \mod \langle f\rangle  . 
%\]
Since the ideals $\mathcal J_n(f)$ and $\mathcal J_n(uf)$ are generated by $\{\det (M_\iota^f)\}_{\iota}$ and  $\{\det (M_\iota^{uf})\}_{\iota}$ respectively, the lemma follows. 
\end{proof}

\begin{definition}
Let $f$ and $g$ be in $\mathbb C\{x\}$. Then $f$ is said to be {\it contact equivalent} to $g$ (at $\mathbf 0$) if there exist an automorphism $\varphi$ of $\mathbb C\{x\}$ and a unit $u$ in $\mathbb C\{x\}$ such that $g=u\cdot \varphi(f)$.	
\end{definition}

\begin{theorem}\label{local-thm1}
Let $f$ and $g$ be in $\mathbb C[x]$ with $f(\mathbf 0)=g(\mathbf 0)=0$. If $f$ is contact equivalent to $g$ at $\mathbf 0$, then $\mathbb C\{x\}/\langle f,\JI_n(f)\rangle$ is isomorphic to $\mathbb C\{x\}/\langle g,\JI_n(g)\rangle$ as $\mathbb C$-algebras for any $n\in \mathbb N^*$.	
\end{theorem}

\begin{proof}
By the hypothesis, there exists an automorphism $\varphi$ of $\mathbb C\{x\}$ and a unit $u$ in $\mathbb C\{x\}$ such that $g=u\cdot \varphi(f)$. Take any $n\in \mathbb N^*$. It follows that
\begin{align*}
\mathbb C\{x\}/\langle g,\JI_n(g)\rangle &=\mathbb C\{x\}/\langle \varphi(f),\JI_n(u\cdot \varphi(f))\rangle\\
&= \mathbb C\{x\}/\langle \varphi(f),\JI_n(\varphi(f))\rangle \quad (\text{by Lemma \ref{local-lem3}})\\
&= \mathbb C\{x\}/\langle \varphi(f),\varphi(\JI_n(f))\rangle \quad (\text{by Lemma \ref{local-lem1}})\\
&=\mathbb C\{x\}/\varphi(\langle f,\JI_n(f)\rangle).
\end{align*} 
Since $\varphi$ is an automorphism $\varphi$ of $\mathbb C\{x\}$, the well defined map
$$\mathbb C\{x\}/\langle f,\JI_n(f)\rangle \to \mathbb C\{x\}/\varphi(\langle f,\JI_n(f)\rangle)$$
which sends $h+\langle f,\JI_n(f)\rangle$ to $\varphi(h)+\varphi(\langle f,\JI_n(f)\rangle)$ is an isomorphism of $\mathbb C$-algebras.	
\end{proof}

\begin{corollary}[\cite{HMYZ23}, Conjecture 1.5]\label{conjj}
Let $f$ and $g$ be in $\mathbb C[x]$ such that $V(f)$ and $V(g)$ have an isolated singularity at $\mathbf 0$. If $f$ is contact equivalent to $g$ at $\mathbf 0$, then $\mathbb C\{x\}/\langle f,\JI_n(f)\rangle$ is isomorphic to $\mathbb C\{x\}/\langle g,\JI_n(g)\rangle$ as $\mathbb C$-algebras for any $n\in \mathbb N^*$.	
\end{corollary}

We remark that Conjecture 1.5 of \cite{HMYZ23} is proved by its authors for $d=n=2$, see \cite[Theorem A]{HMYZ23}. Furthermore, Theorem \ref{local-thm1} is also stronger than this conjecture since it does not need the condition that $f$ and $g$ have an isolated singularity at $\mathbf 0$.

%***********************

\section{Motivic zeta functions of regular functions}

\subsection{Motivic zeta functions}\label{Sec3.1}
For $m\in \mathbb N^*$, we denote by $\mu_m$ the group scheme of $n$-th roots of unity $\Spec\left(k[\tau]/(\tau^m-1)\right)$. These schemes together with the mappings $\mu_{ml}\to \mu_m$ given by $\xi\mapsto \xi^l$ form a projective system, whose limit is denoted by $\hat{\mu}$. A {\it good} $\mu_m$-action on a $k$-variety $X$ is a group action of $\mu_m$ on $X$ such that every orbit is contained in an affine subvariety, a {\it good} $\hat{\mu}$-action is an action of $\hat{\mu}$ factoring via a good $\mu_m$-action. %In the present article, all actions mentioned will be good ones.

Let $S$ be a $k$-variety, with trivial $\mu_m$-action. The {\it $\mu_m$-equivariant Grothendieck group} $K_0^{\mu_m}(\Var_S)$ is the quotient of the free abelian group generated by the $\mu_m$-equivariant isomorphism classes $[X\to S,\sigma]$ by the subgroup generated by
$$[X\to S,\sigma]-[Y\to S,\sigma|_Y]-[X\setminus Y\to S,\sigma|_{X\setminus Y}]$$
for any invariant Zariski closed subvariety $Y$ of $X$, and by 
\begin{align*}%\label{equiv}
[X\times_k\mathbb A_k^n\to S,\sigma]-[X\times_k\mathbb A_k^n\to S,\sigma']
\end{align*}
for $\sigma$ and $\sigma'$ lifted from the same $\mu_m$-action on $X$. There is a natural structure of a commutative ring with unity on $K_0^{\mu_m}(\Var_S)$ due to fiber product over $S$, with $\mu_m$-action on the fiber product induced from the diagonal one. Let $\L$ be the class of the trivial line bundle over $S$. We define $\mathscr M_S^{\mu_m}:=K_0^{\mu_m}(\Var_S)[\L^{-1}]$, $K_0^{\hat \mu}(\Var_S):=\varinjlim K_0^{\mu_m}(\Var_S)$, and $\mathscr M_S^{\hat \mu}:= \varinjlim \mathscr M_S^{\mu_m}=K_0^{\hat \mu}(\Var_S)[\L^{-1}]$. Forgetting actions recovers the (classical) Grothendieck ring $\mathscr M_S$. 

To a $k$-variety $X$ and $m\in \mathbb N$ corresponds the $k$-scheme $\mathscr L_m(X)$ that represents the functor 
$$K\mapsto \Mor_{k-\text{schemes}}(\Spec(K[t]/\langle t^{m+1}\rangle K[t]),X)$$
from the category of $k$-algebras to the category of sets. For integers $l\geq m\geq 0$, the truncation map $k[t]/\langle t^{l+1}\rangle\to k[t]/\langle t^{m+1}\rangle$ induces a morphism of $k$-schemes
$$\pi_m^l: \mathscr{L}_l(X)\to \mathscr L_m(X).$$
If $X$ is smooth of pure dimension $d$, $\pi_m^l$ is a locally trivial fibration with fiber $\mathbb A_k^{(l-m)d}$. Let $\mathscr L(X)$ be the projective limit of $\mathscr L_m(X)$ in the category of $k$-schemes, and let $\pi_m$ be the natural morphism $\mathscr L(X)\to\mathscr L_m(X)$. %For any field extension $K\supseteq k$ we have $\mathscr L(X)(K)=X(K[[t]])$.  

Let $X$ be a smooth $k$-variety of pure dimension $d$, and let $f: X\to\mathbb A_k^1$ be a non-constant function with the scheme-theoretic zero locus $X_0$. For any integer $m\geq 1$, put
$$
\mathscr X_m(f):=\left\{\psi\in \mathscr L_m(X)\mid f(\psi)=t^m\!\!\mod  t^{m+1}\right\}
$$
which is a $k$-variety endowed with the good $\mu_m$-action as below, for $\xi\in \mu_m$,
$$\xi\cdot\psi(t)=\psi(\xi t).$$ 
This together with the natural morphism $\mathscr X_m(f)\to X_0$ yields an element $\big[\mathscr X_m(f)\big]$ in $\mathscr M_{X_0}^{\mu_m}$. 

\begin{definition}
The formal power series $Z_f(T)=\sum_{m\geq 1}\big[\mathscr X_m(f)\big]\L^{-dm}T^m$ is called the {\it motivic zeta function} of $f$.
\end{definition}

Let $\mathscr M_{X_0}^{\hat \mu}[[T]]_{\sr}$ be the $\mathscr M_{X_0}^{\hat \mu}$-submodule of $\mathscr M_{X_0}^{\hat \mu}[[T]]$ generated by 1 and by finite products of elements of the form
$$\frac{\L^pT^q}{1-\L^pT^q}:=\sum_{m\geq 1}(\L^pT^q)^m, \quad (p,q)\in \mathbb Z\times\mathbb N^*.$$ 
An element in $\mathscr M_{X_0}^{\hat \mu}[[T]]_{\sr}$ is called a {\it rational series}. According to \cite{DL98}, there is a unique $\mathscr M_{X_0}^{\hat \mu}$-linear morphism $\lim\limits_{T\to\infty}: \mathscr M_{X_0}^{\hat \mu}[[T]]_{\sr}\to \mathscr M_{X_0}^{\hat \mu}$ with $\lim\limits_{T\rightarrow\infty}\frac{\L^pT^q}{1-\L^pT^q}=-1$. 

By a {\it log-resolution of $(X,X_0)$}, we mean a proper birational morphism $h: Y\to X$ such that $Y$ is smooth, $h$ is an isomorphism over $(X\setminus X_0) \cup (X_0)_{\mathrm{sm}}$, and $h^{-1}(X_0)$ has simple normal crossing support.  Consider a log-resolution $h:Y\to X$ of $(X,X_0)$. Let $E_i$, $i\in J$, be all the irreducible components of $h^{-1}(X_0)$. Assume that
\begin{align*}
\div(h^*f)=\sum_{i\in J}N_iE_i,\quad K_{Y/X}=\sum_{i\in J}(\nu_i-1)E_i,
\end{align*} 
where $N_i$ and $\nu_i$ are positive integers. For $I\subseteq J$, put $E_I=\bigcap_{i\in I}E_i$ and $E_I^{\circ}=E_I\setminus\bigcup_{j\not\in I}E_j$. If $I=\{i\}$, we write $E_i^{\circ}$ instead of $E_{\{i\}}^{\circ}$. Let $U$ be an affine Zariski open subset of $Y$ such that $U\cap E_I^{\circ}\not=\emptyset$, and on it,  
$$h^*f (y)= u(y)\prod_{i\in I}y_i^{N_i},$$ 
where for each $i$, $y_i=0$ is a local equation on the chart $(U,y)$ defining $E_i$, and $u(y)\not=0$ on $U$. Let $N_I$ be the greatest common divisor of $(N_i)_{i\in I}$. Denef-Loeser in \cite{DL02} construct an unramified Galois covering $\pi_I:\widetilde{E}_I^{\circ}\to E_I^{\circ}$, with Galois group $\mu_{N_I}$ given over $U\cap E_I^{\circ}$ by 
\begin{align*}
\widetilde{E}_I^{\circ}|_{U\cap E_I^{\circ}}:=\left\{(z,y)\in \mathbb{A}_k^1\times(U\cap E_I^{\circ}) \mid z^{N_I}=u(y)^{-1}\right\}&\to U\cap E_I^{\circ}\\
(z,y)&\mapsto y.
\end{align*}
Choose a covering of $Y$ by affine open subvarieties $U$. Then the varieties $\widetilde{E}_I^{\circ}|_{U\cap E_I^{\circ}}$ are naturally glued together into a unramified Galois covering $\widetilde{E}_I^{\circ}$, which is endowed with a natural $\mu_{N_I}$-action. The $\mu_{N_I}$-equivariant morphism $\widetilde{E}_I^{\circ}\to E_I^{\circ}\to X_0$ determines a class $\big[\widetilde{E}_I^{\circ}\big]$ in $\mathscr M_{X_0}^{\mu_{N_I}}$. %If $I$ contains an $i$ such that $h(E_i)=\{x\}$, we denote by $\big[\widetilde{E}_I^{\circ}\cap h^{-1}(x)\big]$ the class of $\widetilde{E}_I^{\circ}$ in $\mathscr M_k^{\mu_{N_I}}$ 

\begin{theorem}[Denef-Loeser \cite{DL02}]\label{DL2002}
Using a log-resolution $h$ of $(X,X_0)$, and the previous notation, we have
$$
\big[\mathscr X_m(f)\big]=\L^{md}\sum_{\emptyset\not=I\subseteq J}(\L-1)^{|I|-1}\big[\widetilde{E}_I^{\circ}\big]\left(\sum_{\begin{smallmatrix} k_i\geq 1, i\in I\\ \sum_{i\in I}k_iN_i=m \end{smallmatrix}}\L^{-\sum_{i\in I}k_i\nu_i}\right)
$$
and
$$
Z_f(T)=\sum_{\emptyset\not=I\subseteq J}(\mathbb{L}-1)^{|I|-1}\big[\widetilde{E}_I^{\circ}\big]\prod_{i\in I}\frac{\L^{-\nu_i}T^{N_i}}{1-\L^{-\nu_i}T^{N_i}}.
$$
In particular, $Z_f(T)$ is a rational series.
\end{theorem}

\begin{definition}
The motivic quantity 
\[
\mathscr S_f:=-\lim\limits_{T\to\infty}Z_f(T)=\sum_{\emptyset\not=I\subset J} (1-\L)^{|I|-1}\big[\widetilde{E}_I^{\circ}\big]\ \in \mathscr M_{X_0}^{\hat\mu}
\] 
is called the {\it motivic nearby cycle} of $f$.
\end{definition}

\begin{definition}
Keeping the notation above, for $m\in \mathbb N^*$, we say that a log-resolution $h$ is called {\it $m$-separating} if $N_i+N_j>m$ whenever $E_i\cap E_j\not=\emptyset$. 
\end{definition}

We note that an $m$-separating log-resolution always exists. 
This was proved in \cite[Lemma 2.9]{BFLN22} over $\mathbb C$ with a slight different definition of log-resolutions; they do not assume that a log-resolution $h: Y \to X$ needs to be an isomorphism over $(X\setminus X_0)\cup (X_0)_{\mathrm{sm}}$. However the same proof shows that the result holds over any field of characteristic zero with our definition of log-resolutions. 
 If $h$ is an $m$-separating log-resolution of $(X,X_0)$, it follows from Theorem \ref{DL2002} that
\begin{equation}\label{contactloci}
\big[\mathscr X_m(f)\big]=\L^{md}\sum_{N_i|m}\big[\widetilde{E}_i^{\circ}\big]\L^{-\frac{m\nu_i}{N_i}}.
\end{equation}

%----------------------

\subsection{Motivic zeta function depends on second order Jacobian ideal}
Let $X$ be a smooth $k$-variety of pure dimension $d\geq 2$. If $f$ is a non-constant regular function on $X$, we also denote by $f$ the corresponding element in $\mathcal O_X(X)$.
In what follows, we say that a \(\overline{k}\)-point \(x\) of \(X_0\) is a \textit{node} if its complete local ring \(\widehat{\mathcal O}_{(X_0)_{\overline{k}},x}\) is isomorphic to \(\overline{k}[[x_1,x_2]]/\langle x_1^2 +x_2^2\rangle\). We say that \(X_0\) \textit{has nodes} if it has \(\overline{k}\)-points which are nodes.

\begin{theorem}\label{theorem1}
Let $f$ and $g$ be non-constant regular functions on $X$ with the same scheme-theoretic zero locus $X_0$. Suppose that $g-f\in \JI_2(f)$. If $d=2$ and $X_0$ has nodes, suppose additionally that $k$ is quadratically closed. Then, for any integer $m\geq 1$, the identity $\big[\mathscr X_m(f)\big]=\big[\mathscr X_m(g)\big]$ holds in $\mathscr M_{X_0}^{\mu_m}$. As a consequence, $Z_f(T)=Z_g(T)$ and $\mathscr S_f=\mathscr S_g$.
\end{theorem}

\begin{remark}
Assume that $f$ and $g$ have the same zero locus, and 
\begin{itemize}
	\item[(a)] $g-f\in \JI_2(f)$ for $d\geq 3$,
	
	\item[(b)] $g-f\in \JI_3(f)$ for $d=2$.
\end{itemize}
By Proposition \ref{JnJ1}, we have $\JI_2(f)\subseteq \JI_1(f)^3$ if $d\geq 3$, and $\JI_3(f)\subseteq \JI_1(f)^3$ if $d=2$. This together with (a) or (b) implies that $g-f\in \JI_1(f)^3$. Then we deduce from (a variant of) \cite[Theorems 3.2, 3.6]{BJM19} that $\mathscr S_f=\mathscr S_g$. In our proof of Theorem \ref{theorem1} mentioned below, the idea can be applied to provide a new proof for \cite[Theorems 3.2, 3.6]{BJM19} using log-resolution.
\end{remark}

\begin{proof}[{\bf Proof of Theorem \ref{theorem1}}]
Let us write $g-f=a$ for some $a\in \JI_2(f)$. Since $f$ and $g$ have the same zero locus $X_0$, we have $a|_{X_0}=0$. Write locally, on $W$, $a=a'f$, where $W$ is an open affine subset of $X$ and $a'\in \mathcal O_X(W)$. Let $\overline{x}$ be a non-smooth point of $X_0\cap W$. Up to \'etale base change, we may assume that $W=\mathbb A_k^d$. Let $N\geq 2$ be the multiplicity of $f$ at $\overline{x}$ so that we can write $f$ as
$$f(x)=f_N(x)+ \text{(terms of degree $>N$)},$$
using local coordinates $x=(x_1,\dots,x_d)$ and a nonzero homogeneous polynomial $f_N$ of degree $N$. First, we consider the following three cases:
\begin{itemize}
	\item[(i)] $d\geq 3$,
	
	\item[(ii)] $d=2$ and $N>2$,
	
	\item[(iii)] $d=N=2$ and $f_2(x_1,x_2)=cx_1^2$, $c\in k^*$ (up to linear change of variables).
\end{itemize}
We are going to prove that $a'(\overline{x})=0$. Let $\alpha=\alpha(\overline{x})$ be a vector in $\mathbb N^d$ such that $|\alpha|=N$ and $(\partial^{\alpha} f)(\overline{x})\not=0$. By the general Leibniz rule, applying to the equality $a=a'f$ on $W=\mathbb A_k^d$,
\begin{equation}\label{partiala}
\partial^{\alpha}a=a'\partial^{\alpha}f+\sum_{\alpha'>0,\ \alpha'+\beta=\alpha}\frac{\alpha!}{\alpha'!\beta!}(\partial^{\alpha'}a')(\partial^{\beta}f).
\end{equation}	
In the sum $\sum$ on the right hand side, since $\alpha'>0$, we have $|\beta|<|\alpha|=N$, hence $(\partial^{\beta}f)(\overline{x})=0$. If $d\geq 3$, since $a$ is a section of $(\JI_1)^3$, we can write $a$ in the form
$$a=\sum_{1\leq i,j,l\leq d}a_{ijl}\frac{\partial f}{\partial x_i}\frac{\partial f}{\partial x_j}\frac{\partial f}{\partial x_l},$$
where $a_{ijl}$'s are sections of $\mathcal O_X$. By the general Leibniz rule,
$$\partial^{\alpha}a=\sum_{1\leq i,j,l\leq d}\ \sum_{\alpha'+\beta=\alpha}\frac{\alpha!}{\alpha'!\beta!}(\partial^{\alpha'}a_{ijl})\left(\partial^{\beta}\left(\frac{\partial f}{\partial x_i}\frac{\partial f}{\partial x_j}\frac{\partial f}{\partial x_l}\right)\right).$$
Let $e_i$ be the $i$-th vector of the canonical basis of $\mathbb N^d$, let $\gamma(i)=\gamma+e_i$ for any $\gamma\in \mathbb N^d$. Then, we have
$$\partial^{\beta}\left(\frac{\partial f}{\partial x_i}\frac{\partial f}{\partial x_j}\frac{\partial f}{\partial x_l}\right)=\sum_{\beta'+\beta''+\beta'''=\beta}\frac{\beta!}{\beta'!\beta''!\beta'''!}(\partial^{\beta'(i)}f)(\partial^{\beta''(j)}f)(\partial^{\beta'''(l)}f).$$
If all the three numbers $|\beta'(i)|$, $|\beta''(j)|$ and $|\beta'''(l)|$ are simultaneously $\geq N$, then $N+3\geq |\beta|+3=|\beta'(i)|+|\beta''(j)|+|\beta'''(l)|\geq 3N$, so $N\leq \frac 3 2$, which is impossible. Thus, at least one of these numbers is strictly less than $N$, for instance, $|\beta'(i)|<N$, hence $(\partial^{\beta'(i)}f)(\overline{x})=0$. It follows that 
\begin{equation}\label{three1stpd}
\partial^{\beta}\left(\frac{\partial f}{\partial x_i}\frac{\partial f}{\partial x_j}\frac{\partial f}{\partial x_l}\right)(\overline{x})=0
\end{equation}
for every $\beta\leq \alpha$, hence $(\partial^{\alpha}a)(\overline{x})=0$. Thus, it follows from (\ref{partiala}) that $a'(\overline{x})=0$. 

If $d=2$, then $\JI_2(f)$ has a system of generators consisting of
$$\bigg(\frac{\partial f}{\partial x_1}\bigg)^3,\ \ \bigg(\frac{\partial f}{\partial x_1}\bigg)^2 \frac{\partial f}{\partial x_2},\ \ \frac{\partial f}{\partial x_1} \bigg(\frac{\partial f}{\partial x_2}\bigg)^2,\ \ \bigg(\frac{\partial f}{\partial x_2}\bigg)^3,$$
and
$$\sigma:=\frac 1 2 \frac{\partial^2f}{\partial x_1^2} \bigg(\frac{\partial f}{\partial x_2}\bigg)^2-\frac{\partial^2f}{\partial x_1\partial x_2} \frac{\partial f}{\partial x_1} \frac{\partial f}{\partial x_2}+\frac 1 2 \frac{\partial^2f}{\partial x_2^2} \bigg(\frac{\partial f}{\partial x_1}\bigg)^2.$$
By (\ref{three1stpd}), in order to prove $(\partial^{\alpha}a)(\overline{x})=0$, hence $a'(\overline{x})=0$, it suffices to prove that 
$$(\partial^{\beta}\sigma)(\overline{x})=0$$
for every $\beta\leq \alpha$. By the general Leibniz rule, 
$$\partial^{\beta}\left(\frac{\partial f}{\partial x_i}\frac{\partial f}{\partial x_j}\frac{\partial^2 f}{\partial x_l\partial x_p}\right)=\sum_{\beta'+\beta''+\beta'''=\beta}\frac{\beta!}{\beta'!\beta''!\beta'''!}(\partial^{\beta'(i)}f)(\partial^{\beta''(j)}f)(\partial^{\beta'''(l)(p)}f).$$
If $|\beta'(i)|\geq N$, $|\beta''(j)|\geq N$ and $|\beta'''(l)(p)|\geq N$, then $N+4\geq |\beta|+4\geq 3N$, hence $N\leq 2$. Thus, if $N>2$, the contradiction argument (see a similar detail as in the case $d\geq 3$) shows that $(\partial^{\beta}\sigma)(\overline{x})=0$ for every $\beta\leq \alpha$. If $f_2(x_1,x_2)=cx_1^2$, $c\in k^*$, we may choose $\alpha=(2,0)$. Assume for simplicity that $\overline{x}=(0,0)$. Then, we have $\sigma=0+(\text{terms of degree}\ \geq 3)$, hence $\sigma(0,0)=\frac{\partial \sigma}{\partial x_1}(0,0)=\frac{\partial^2 \sigma}{\partial x_1^2}(0,0)=0$.

Now, we take care of the fourth case
\begin{itemize}
	\item[(iv)] $d=N=2$ and $f_2(x_1,x_2)=x_1^2+x_2^2$ (up to linear change of variables).
\end{itemize}
In this case, $\overline{x}$ is a node of $X_0\cap W$, and for simplicity we assume $\overline{x}=(0,0)$. Then, $f$ and $\sigma$ have the forms $f(x_1,x_2)=x_1^2+x_2^2+(\text{terms of degree}\ \geq 3)$, $\sigma=4(x_1^2+x_2^2)+(\text{terms of degree}\ \geq 3)$. Since $(0,0)$ is also a nodal zero of $g$, we have
$g(x_1,x_2)=c(x_1^2+x_2^2)+(\text{terms of degree}\ \geq 3)$ for some $c\in k^*$.

\medskip 
To finish the proof we are going to consider the following two cases.

\medskip 
{\bf Case I: All singular points of $X_0$ satisfy (i), (ii) and (iii).} Let $h: Y\to X$ be an $m$-separating log-resolution of $(X,X_0)$. We write 
$$\div(h^*f)=\sum_{i\in J}N_i(f)E_i, \quad  K_{Y/X}=\sum_{i\in J}(\nu_i-1)E_i,$$ 
where $E_i$, $i\in J$, are irreducible components of $h^{-1}(X_0)$. 
Assume that $E_i$ is an exceptional divisor. Let $U$ be an open affine subscheme of $Y$ with $U\cap E_i^{\circ}\not=\emptyset$, so that on $U\cap E_i^{\circ}$ we have 
$$h^*f= uy_i^{N_i(f)},$$ 
with $y_i=0$ being the local equation of $E_i$, and $u(y)\not=0$ for every $y\in U\cap E_i^{\circ}$. Since $g-f=a$, we have $(h^*g)|_U-(h^*f)|_U=(h^*a)|_U$. It implies that, shrinking $U$ if necessary, 
\begin{align}\label{eq1}
(h^*g)|_U-(h^*f)|_U \in  \Big\langle y_i^{\ord_{E_i}(h^*a)}\Big\rangle \subseteq \mathcal O_Y(U),
\end{align}
from which
\begin{align*}%\label{eq2}
(h^*g)|_U \in (h^*f)|_U + \Big\langle y_i^{\ord_{E_i}(h^*a)}\Big\rangle.
\end{align*}
Since $a=a'f$ on $W$ and $a'(\overline{x})=0$, we have $\ord_{E_i}(h^*a)>N_i(f)$. Therefore, we deduce that there exists a morphism $v: U\to \mathbb G_m$ such that 
$$(h^*g)|_U=vy_i^{N_i(f)}.$$
This argument works for every exceptional divisor $E_i$, $i\in J$, hence $h: Y\to X$ is also an $m$-separating log-resolution of $(X,X_0)$ {\it for $g$} with the same exceptional divisors $E_i$'s as for $f$ and $N_i(f)=N_i(g)=:N_i$ on each $E_i$. Thus, by (\ref{eq1}) we have
$$u(y)-v(y)\in \Big\langle y_i^{\ord_{E_i}(h^*a)-N_i}\Big\rangle,$$
so $u=v$ on $U\cap E_i^{\circ}$. Consider the Galois unramified coverings $\widetilde{E}_i^{\circ}(f)$ and $\widetilde{E}_i^{\circ}(g)$ described over $U\cap E_i^{\circ}$ as follows
$$\widetilde{E}_i^{\circ}(f)|_{U\cap E_i^{\circ}}=\left\{(z,y)\in \mathbb{A}_k^1\times(U\cap E_i^{\circ}) \mid z^{N_i}=u(y)^{-1}\right\}$$
and
$$
\widetilde{E}_i^{\circ}(g)|_{U\cap E_i^{\circ}}=\left\{(z,y)\in \mathbb{A}_k^1\times(U\cap E_i^{\circ}) \mid z^{N_i}=v(y)^{-1}\right\}.$$ 
Hence $\widetilde{E}_i^{\circ}(f)|_{U\cap E_i^{\circ}}=\widetilde{E}_i^{\circ}(g)|_{U\cap E_i^{\circ}}$ because $u=v$ on $U\cap E_i^{\circ}$. 

Let $U'$ be another open affine subscheme of $Y$ such that $U'\cap  E_i^{\circ}\not=\emptyset$ on which $h^*f= u'z_i^{N_i}$,  
with $z_i=0$ defining $U'\cap E_i$, and $u'$ a unit. Similarly as previous, $h^*g=v'z_i^{N_i}$ on $U'$ for some unit $v'$, and furthermore, $\widetilde{E}_i^{\circ}(f)|_{U'\cap E_i^{\circ}}=\widetilde{E}_i^{\circ}(g)|_{U'\cap E_i^{\circ}}$  Then, on $U\cap U'$, $z_i=\xi_iy_i$ with $\xi_i$ a unit, hence $u=u'\xi_i^{N_i}$ and $v=v'\xi_i^{N_i}$; thus the map 
$$\left\{(z,y)\in \mathbb{A}_k^1\times(U\cap U'\cap E_i^{\circ}) \mid z^{N_i}=u(y)^{-1}\right\}\to \left\{(z,y)\in \mathbb{A}_k^1\times(U\cap U'\cap E_i^{\circ}) \mid z^{N_i}=u'(y)^{-1}\right\}$$
sending $(z,y)$ to $(\xi_i z,y)$ is an isomorphism, the same observation also holds for $v$ and $v'$. It follows that
$$\widetilde{E}_i^{\circ}(f)|_{U\cap U'\cap E_i^{\circ}}=\widetilde{E}_i^{\circ}(g)|_{U\cap U'\cap E_i^{\circ}}.$$
hence the gluing yields $\widetilde{E}_i^{\circ}(f)=\widetilde{E}_i^{\circ}(g)$.

Since $X_0(f)=X_0(g)=X_0$ and, again, $u=v$ on $U\cap E_i^{\circ}$ for every appropriate $U$ and $i$, it follows that if $E_j$ is a strict transform for $f$, it is also a strict stransform for $g$, thus $\widetilde{E}_j^{\circ}(f)=E_j^{\circ}=\widetilde{E}_j^{\circ}(g)$.

Since $h$ is an $m$-separating log-resolution of $(X,X_0)$ common for $f$ and $g$, the numerical invariants $\nu_i$ are common for $f$ and $g$. Using the formula (\ref{contactloci}) we have
$$\big[\mathscr X_m(f)\big]=\L^{md}\sum_{N_i|m}\big[\widetilde{E}_i^{\circ}(f)\big]\L^{-\frac{m\nu_i}{N_i}}$$
and
$$\big[\mathscr X_m(g)\big]=\L^{md}\sum_{N_i|m}\big[\widetilde{E}_i^{\circ}(g)\big]\L^{-\frac{m\nu_i}{N_i}}.$$
Therefore, by the above discussions, we get $\big[\mathscr X_m(f)\big]=\big[\mathscr X_m(g)\big]$.	

\medskip 
{\bf Case II: $d=2$ and $X_0$ contains a node, i.e. there is a point $\overline{x}\in X_0$ satisfying (iv).}  Let $h: Y\to X$ be a log-resolution of $(X,X_0)$ for $f$. As in the proof of \cite[Lemma 2.9]{BFLN22} (applying to the case $d=2$), given $h$, we can blow up along intersections $E_l\cap E_s$, and do this procedure many times to attain the property $N_i+N_j>m$ whenever $E_i\cap E_j\not=\emptyset$ and $E_i$ and $E_j$ come from non-nodal singular points (i.e. $m$-separating for non-nodal singular points). For each node of $X_0$, we only use once the blowup $(y_1,y_2)\mapsto (y_1y_2,y_2)$, locally. The resulting $m$-separating log-resolution for non-nodal singular points is also denoted by $h$. Using the argument in Case I, it remains to prove that, if $E_i$ is the exceptional prime divisor corresponding to the node $\overline{x}=(0,0)$ of $X_0$, then $\widetilde{E}_i^{\circ}(f)\cong \widetilde{E}_i^{\circ}(g)$.

As explained in (iv), $f(x_1,x_2)=x_1^2+x_2^2+(\text{terms of degree}\ \geq 3)$ and $g(x_1,x_2)=c(x_1^2+x_2^2)+(\text{terms of degree}\ \geq 3)$, $c\in k^*$. We have, locally, say, on $U$,
$$h_1^*f(y_1,y_2)=uy_2^2, \quad u=1+y_1^2+y_2\cdot (\text{terms of degree} \geq 0)$$
and
$$h_1^*g(y_1,y_2)=cvy_2^2, \quad v=1+y_1^2+y_2\cdot (\text{terms of degree} \geq 0).$$
Here, $U\cap E_i$ is defined by $y_2=0$ and $N_i(f)=N_i(g)=2$. Since $u=v=1+y_1^2$ on $U\cap E_i$ and $c\not=0$, the map
$$\left\{(z,y_1,0)\in \mathbb{A}_k^1\times(U\cap E_i^{\circ}) \mid z^2=(1+y_1^2)^{-1}\right\}\to \left\{(z,y_1,0)\in \mathbb{A}_k^1\times(U\cap E_i^{\circ}) \mid cz^2=(1+y_1^2)^{-1}\right\}$$
sending $(z,y_1,0)$ to $(\frac{z}{\sqrt{c}},y_1,0)$ is an isomorphism. Therefore, by gluing we obtain $\widetilde{E}_i^{\circ}(f)\cong \widetilde{E}_i^{\circ}(g)$.
\end{proof}

\begin{corollary}
Let $f$ and $g$ be non-constant regular functions on $X$ with the same scheme-theoretic zero locus $X_0$. Suppose that $d$, $n$, $f$ and $g$ satisfy one of the following conditions:
\begin{itemize}
	\item $d\geq 3$, $n\geq 2$, $g-f\in \JI_n(f)$,
	
	\item $d=2$, $n\geq 2$, $g-f\in \JI_n(f)$, $X_0$ has no node,
	
	\item $d=2$, $n\geq 2$, $g-f\in \JI_n(f)$, $X_0$ has nodes, $k$ is quadratically closed,
	
	\item $d=2$, $n\geq 3$, $g-f\in \JI_n(f)$.
\end{itemize}
Then, for any integer $m\geq 1$, the identity $\big[\mathscr X_m(f)\big]=\big[\mathscr X_m(g)\big]$ holds in $\mathscr M_{X_0}^{\mu_m}$. As a consequence, $Z_f(T)=Z_g(T)$ and $\mathscr S_f=\mathscr S_g$.	
\end{corollary}

\begin{ack}
The first author thanks Department of Mathematics, Graduate School of Science, Osaka University for warm hospitality during his visit as a Specially Appointed Researcher supported by the MSJ Tosio Kato Fellowship. He would also like to acknowledge support from the ICTP through the Associates Programme (2020-2025).
\end{ack}

%----------------------------------

%\selectlanguage{english}

\end{document}